\documentclass[11pt]{amsart}
\usepackage{geometry}                
\usepackage{graphicx}
\usepackage{amssymb}
\usepackage{epstopdf}
\newcommand{\Z}{\mathbb{Z}}
\newcommand{\R}{\mathbb{R}}
\newcommand{\C}{\mathbb{C}}

\newcommand{\cM}{\mathcal{M}}
\newcommand{\cP}{\mathcal{P}}
\newcommand{\re}{\operatorname{Re}}
\newcommand{\im}{\operatorname{Im}}
\newcommand{\area}{\operatorname{area}}
\numberwithin{equation}{section}
\numberwithin{figure}{section}

\theoremstyle{plain} 
\newtheorem{theorem}{Theorem}[section]
\newtheorem{lemma}[theorem]{Lemma}
\newtheorem{corollary}[theorem]{Corollary}
\newtheorem{proposition}[theorem]{Proposition}
\theoremstyle{definition} 
\newtheorem{definition}[theorem]{Definition}

\newtheorem{example}[theorem]{Example}
\newtheorem*{question}{Question}
\newtheorem*{notation}{Notation}

\author[Drew Reisinger]{Drew Reisinger}
   \address{University of Evansville}
   \email{Drew Reisinger <drewissimo333@gmail.com>}

\author[Matthias Weber]{Matthias Weber}
   \address{Indiana University}
   \email{matweber@indiana.edu}
    \urladdr{http://www.indiana.edu/~minimal}
  \thanks{Both authors were partially supported by NSF grant DMS-1156515. The second author was partially supported by  grant 246039 from the Simons Foundation}

\date{}  
\title{Moduli of Parallelogram Tilings and Curve Systems}

\begin{document}

\begin{abstract}
We  determine the topology of the moduli space of periodic tilings of the plane by parallelograms. To each such tiling, we associate combinatorial data  via the zone curves of the tiling. We show that all tilings with the same combinatorial data form an open subset in a suitable Euclidean space that is homotopy equivalent to a circle. Moreover, for any choice of combinatorial data, we construct a canonical tiling with these data.
\end{abstract}


\maketitle

\section{Introduction}
The purpose of this paper is to study deformation spaces  of periodic tilings of the plane by parallelograms that are edge-to-edge, and related topics.

We associate to any such tiling a curve system on the quotient torus of the plane by the period lattice of the tiling. The curves are obtained as the zone curves
of the tiling that arise when traversing opposite edges of the parallelograms, very much like one does for  zonohedra (\cite{coxeter}). In addition to these topological data, we associate to each zone curve the vector of the edge that is being traversed as geometric data.

Thus we are able to to separate  the combinatorial information of a periodic tiling from its geometric data. This approach to study deformation spaces was was pioneered by Penner (\cite{penner}), where it leads to a cell decomposition of the Teichmuller space of punctured Riemann surfaces.

In our case we show that the deformation space of periodic tilings for a fixed curve system with $n$ curves is a complex  $n$ dimensional connected open subset of  Euclidean space $\C^n$ that is homeomorphic to an annular neighborhood of a circle. A surprising byproduct of the proof is that we construct a canonical tiling, up to similarity, for any given curve system. In other words, topological tilings of tori by quadrilaterals have canonical Euclidean realizations.

\medskip

Slightly more general  curve systems arise in the the case of \emph{quadrangulations}, i.e. surfaces of higher genus that are topologically tiled by quadrilaterals (see \cite{jsk} for applications in in computer vision).
We show that the topological curve systems that arise in this general setting 
are up to isotopy in one-to-one correspondence with certain simple combinatorial data that encode the intersection pattern of the curve system. This allows for an algorithmic treatment
of the periodic tilings  and also opens the way to look at geometric realizations of higher genus surfaces in Euclidean space, either as closed or as periodic surfaces.

\medskip

 The organization of this paper is as follows: In section 2, we  introduce topological surfaces with curve systems and extract their combinatorial intersection patterns as what we call \emph{combinatorial curve systems}. We then go on to prove that surfaces with curve systems are in one-to-one correspondence with combinatorial curve system under suitable identifications.
 
 In section 3, we turn to periodic tilings of the plane by parallelograms.  We use zones to define a canonical zone  curve system on the quotient torus of the tiled plane by its period lattice. 
 We give a necessary and sufficient condition for a choice of edge data to produce a periodic tiling and show that this condition can always be satisfied.
 
 In section 4, we
fix a  combinatorial curve system of genus 1 and describe the moduli space of periodic tilings with that underlying combinatorial curve system. More concretely, we determine its dimension, topological type, and local boundary structure. The key idea here is to use the intersection matrix of the curve system. It turns out that its eigenvector for a non-zero eigenvalue can be used to define \emph{canonical} edge data. All other edge data for the same combinatorial curve system can be geometrically deformed into the canonical edge data.

\section{Combinatorial Curve Systems for Quadrilateral Tilings of Surfaces}

Let $S$ be an oriented, closed (but not necessarily connected) surface.
\begin{definition}
\label{curve_sys}
A \emph{topological curve system} on  $S$ is a finite list of regular, oriented simple closed curves  $\Gamma = (\gamma_1, \gamma_2, \ldots , \gamma_m)$ on 
$S$, $m \ge 2$, such that 
\begin{enumerate}
\item all intersections between curves are double points,
\item  every curve intersects at least one other curve, and
\item each component of $S \backslash \Gamma$  is a  topological disk (here $\Gamma$ is identified with the trace of its curves).
\end{enumerate}
\end{definition}
\begin{figure}[h]
\begin{center}
\includegraphics[width=0.5\linewidth]{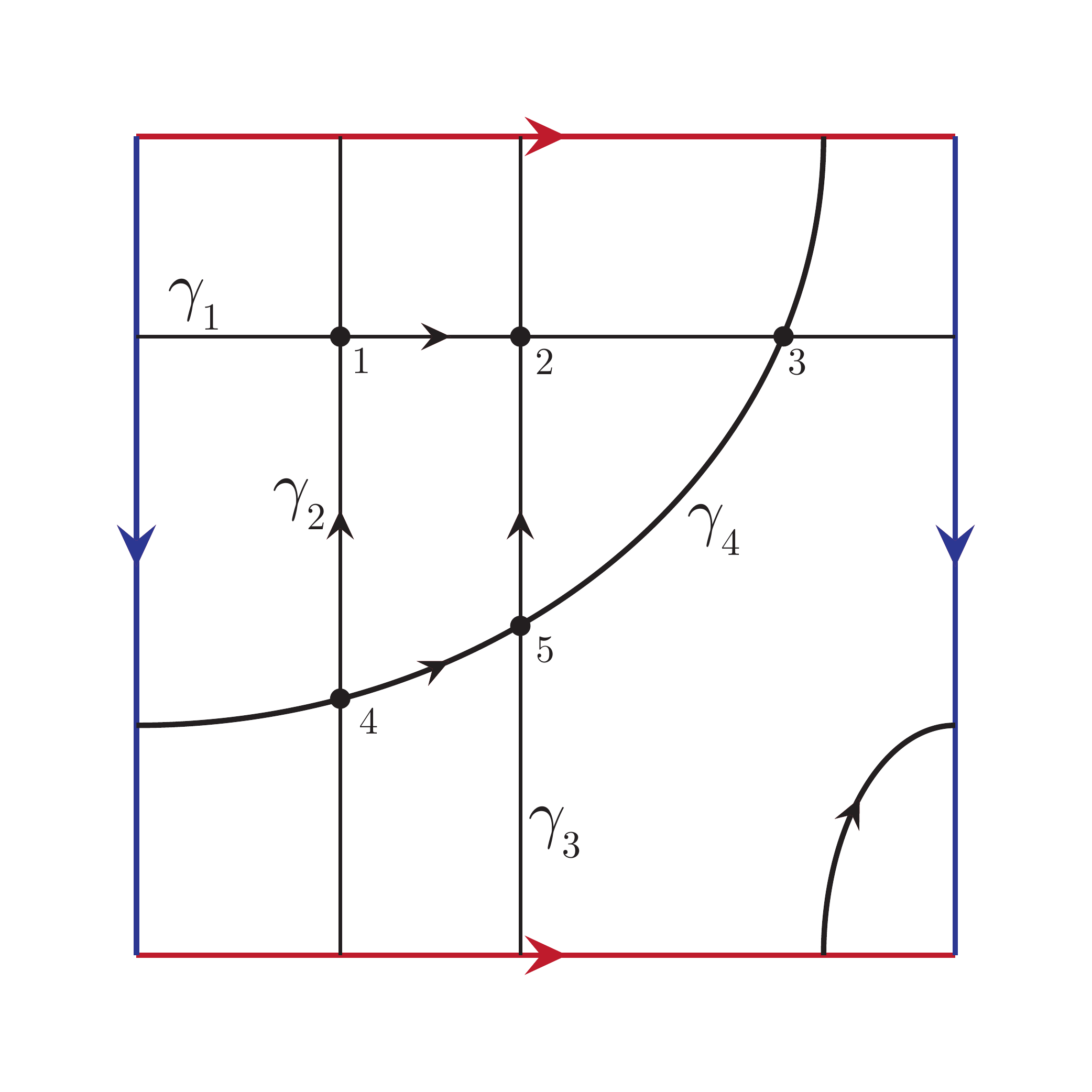}
\caption{An example of a curve system on a torus}
\label{curve_ex}
\end{center}
\end{figure}

\begin{example}
\label{example:curvesystem}
Consider a torus, represented as a square with opposite edges identified.
Figure \ref{curve_ex} depicts a curve system on such a  torus consisting of four curves. Note that each of the system's intersection points
involves exactly two curves; we will explain the integer labels on these intersection points shortly.
\end{example}
%
%
%
Our first goal is a combinatorial description of the \emph{oriented} intersections between curves in a curve system. Label the intersection points of $\Gamma$ with distinct positive integers $k_1, \ldots, k_n$, $n \ge 1$. For each $\gamma_i \in \Gamma$, let 
\[
a_i = (a_i^1, a_i^2, \ldots, a_i^{m_i}), 
\]
$a_i^j \in \{\pm k_1, \pm k_2, \ldots, \pm k_n\},$ be the vector, ordered by the curve's orientation, of signed intersection points on $\gamma_i$. The sign of an entry is 
positive if the intersection 
 of $\gamma_i$ with the other curve at that point is positive, and negative otherwise. We identify these vectors up to cyclic permutation so that this
correspondence is well-defined.

In Example \ref{example:curvesystem}, the encoding of the system in Figure \ref{curve_ex} is
\begin{align*}
a_1 & =  (1,2,3) \\
a_2 & =  (-1,-4) \\
a_3 & = (-2,-5) \\
a_4 & = (-3, 4, 5).
\end{align*}

We have thus defined a procedure for describing a topological curve system $\Gamma$ on a surface $S$ by a choice of
intersection labels $k_1, \ldots, k_n$ and a list of integer vectors, which we call the \emph{encoding}
of the tuple $(S, \Gamma, (k_1, \ldots, k_n))$. 

Our next goal is to reverse this process. To this end, we  need define these lists of vectors as objects more precisely.

%
%
%
\begin{definition}
A \emph{combinatorial curve system} on a set of distinct positive integers $k_1, \ldots, k_n$, $n \ge 1$, is a list of vectors $A = (a_1, \ldots, a_m)$,
$m \ge 2$, with entries from the set $\{\pm k_1, \ldots, \pm k_n\}$ and identified up to cyclic permutation that satisfies the
following conditions:
\begin{enumerate}
\item Each element of $\{\pm k_1, \ldots, \pm k_n\}$ appears exactly once in some vector $a_i$, and
\item $+k_j$ and $-k_j$ never appear in the same vector.
\end{enumerate}
\end{definition}

\begin{example}
Let $S$ be a surface, $\Gamma$  a topological curve system on $S$, 
and $(k_1, \ldots, k_n)$  a list of positive
integer labels for intersection points of $\Gamma$. Then its encoding as described above is
 a combinatorial curve system $A$ on $k_1, \ldots, k_n$. 
\end{example}

Our next goal is to show that this process of encoding can be reversed in a natural way and thus to establish a one-to-one correspondence between topological and combinatorial curve systems up to suitable identifications.

%
\begin{theorem}
\label{bigone}
Let $A$ be a combinatorial curve system on $k_1, \ldots, k_n$ with $n \ge 1$. Then there exists a surface $S$ and a curve system 
$\Gamma$ on $S$ such that the encoding of $(S, \Gamma, (k_1, \ldots, k_n))$ is $A$. This surface and curve system are unique up to an isotopy preserving the curve system and the labeling of the intersection points.
\end{theorem}

We first discuss the intuitive motivation for the construction behind this theorem before we proceed with its 
formal proof.

One way to understand a topological curve system is as a directed graph on the surface $S$ whose vertices correspond
to the intersection points of the system and whose edges correspond to the segments of the curves between intersections (see Figure \ref{graph_ex}). By
the definition of a topological curve system, this graph divides the surface into faces that are homeomorphic to disks. The
essence of the proof of Theorem \ref{bigone} is that a combinatorial curve system provides enough information to
reconstruct the boundary cycle of each face in counter-clockwise order. Gluing disks into the boundaries and identifying these faces along the appropriate edges
will reconstruct the original surface.

\begin{figure}[h]
\begin{center}
\includegraphics[width=0.6\linewidth]{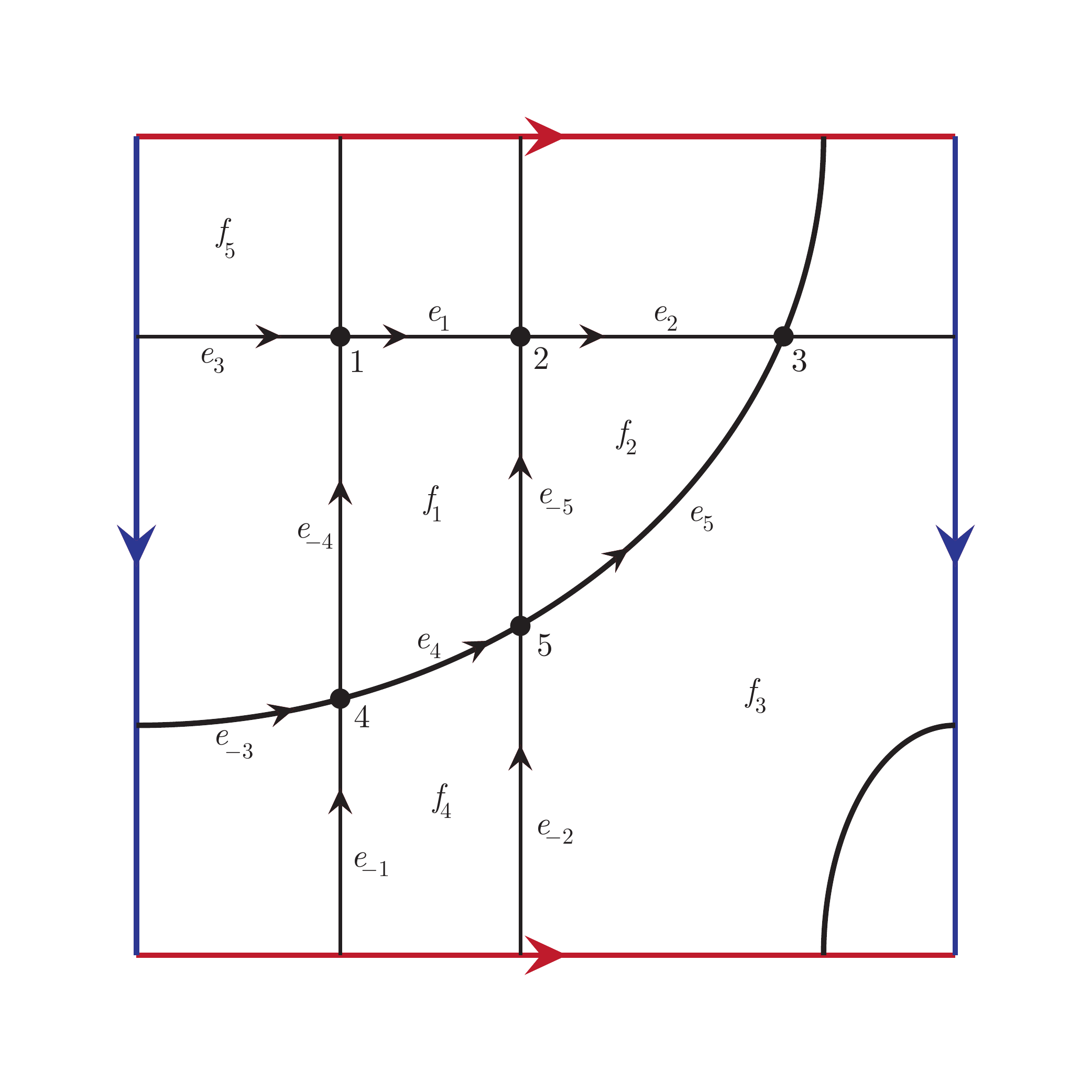}
\caption{The system from Figure \ref{curve_ex} as a graph with its vertices, edges, and faces labeled. Note that
some of the edges and faces continue over the identified edges of the square in this diagram.}
\label{graph_ex}
\end{center}
\end{figure}

To simplify the notation in the formal proof below, we introduce

\begin{notation}
If $a_i$ is a vector in a combinatorial curve system $A$, and $k$ is an integer in $a_i$, we denote the cyclic
predecessor of $k$ in $a_i$ by $k^-$ and the cyclic successor by $k^+$. For example, if $a_1 = (1, -2, 3)$ is a
vector  in some
combinatorial curve system $A$, then $1^+ = -2$, $1^- = 3$, $3^+ =  1$, and $3^- = -2$ in this system.
\end{notation}

Observe that if a vector has only of two entries like $a_2=(-1,-4)$, then $-4$ is both the successor and predecessor of $-1$.

In the course of this proof, we introduce the formalism of \emph{borders}. Borders can be thought of as the result of longitudinally cutting an oriented edge into its left and right halves (see Figure \ref{segs}).

More concretely:

\begin{definition}
Consider the oriented edge between
$k_i$ and its successor $k_i^+$. We associate to each such edge a left border, 
denoted as the (signed) pair $[k_i, k_i^+]$, and a right border, denoted $-[k_i, k_i^+]$.
\end{definition}

The reason for introducing the signed notation  $-[k_i, k_i^+]$ instead of reversing the order $[k_i^+,k_i]$ is due to the fact that vertices can be simultaneously successors and predecessors.

Let $\mathcal{S}$ be the set of all such disjoint borders. We also define a \emph{successor} operation on borders.

\begin{definition}

For each border in $\mathcal{S}$, we define its
\emph{successor} based on which of the following four forms the border takes; recall that $k_i$ denotes a positive integer.
\begin{itemize}
\item If a border is of the form $[k_i^-, k_i]$, then its successor is $[-k_i, (-k_i)^+]$.
\item If a border is of the form $[(-k_i)^-, -k_i]$, then its successor is $-[k_i^-, k_i]$.
\item If a border is of the form $-[k_i, k_i^+]$, then its successor is $-[(-k_i)^-,-k_i]$.
\item If a border is of the form $-[-k_i, (-k_i)^+]$, then its successor is $[k_i, k_i^+]$.
\end{itemize}

\end{definition}

\begin{example}
In Figure \ref{graph_ex}, the edge $e_2$ consists of the two borders  $[4, 5]$ and $-[4,5]$. The successor of $[4, 5] = [5^-, 5]$ is $[-5, (-5)^+] = [-5, -2]$ by the first case above,
while the successor of $-[4, 5] = -[4,4^+]$ is $-[(-4)^-,-4] = -[-1, -4]$ by the third case. The complexity of the notation is solely due to the fact that we need to keep track of the signs of the intersections
between the curves.

\end{example}

\begin{figure}[h]
\begin{center}
\includegraphics[width=0.5\linewidth]{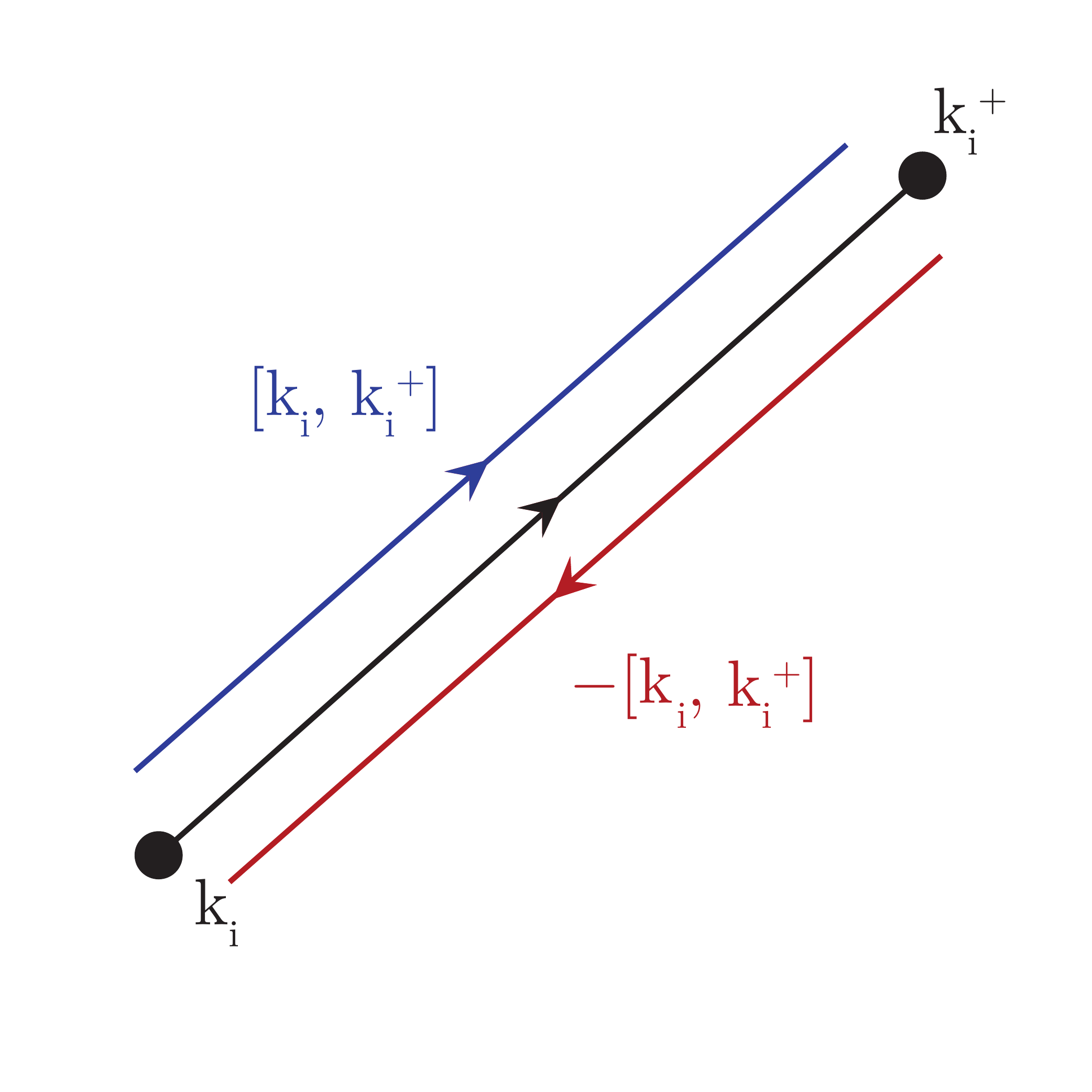}
\caption{Borders represent the left and right sides of an oriented edge.}
\label{segs}
\end{center}
\end{figure}

Note that each border also takes exactly one of the successor forms listed above, so we can also define
the \emph{predecessor} of a border as the inverse of the above operation. Intuitively, the successor of a border can be obtained by following a border to its second endpoint and then turning left at the vertex (see
Figure \ref{succ}).

\begin{figure}[h]
\begin{center}
\includegraphics[width=0.5\linewidth]{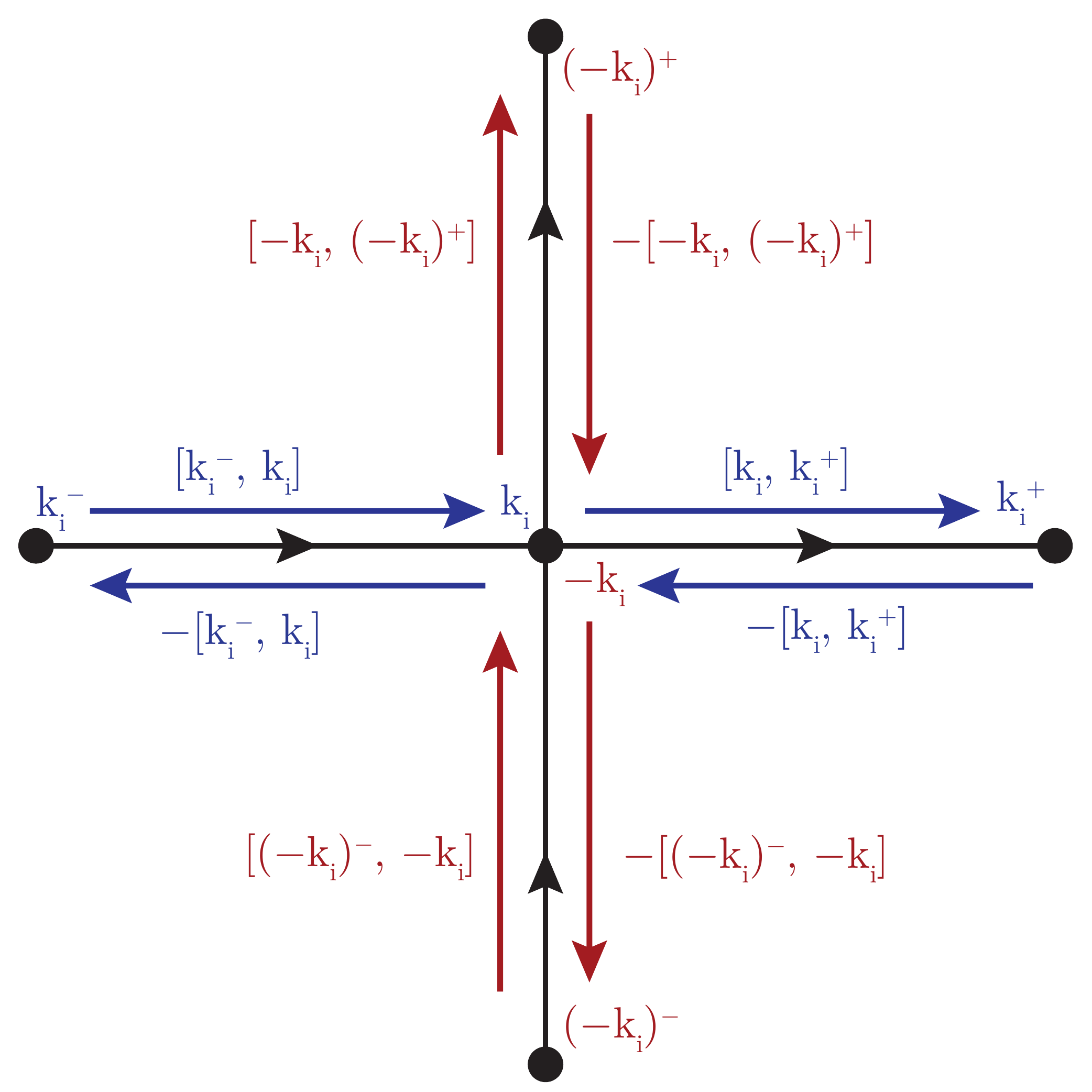}
\caption{The successor operation represents turning
left at a vertex.}
\label{succ}
\end{center}
\end{figure}

\begin{proof}(of Theorem \ref{bigone})
We begin by gluing the borders together to form disjoint oriented loops:

For each border $s \in \mathcal{S}$, denote by $[0,1]_s$ a copy of the unit interval, which we will identify with $s$ for the remainder of the proof.
Let \[\hat L = \bigcup_{s \in \mathcal{S}}[0,1]_s\] be the disjoint union of all borders, and let $L$ be the quotient space of $\hat L$ formed by 
identifying the $1$ endpoint of each border with the $0$ endpoint of its successor. The space $L$ is then
composed of a finite disjoint union of topological circles, which we call \emph{loops} (see Figure \ref{loops}). To see this, consider linking the
successors of a given border $s$. As there are only finitely many borders in $\mathcal{S}$, the end of some
successor border must eventually attach to the beginning of a border that is already in this loop. But since the 
predecessor of each border is unique, only two borders can meet at any one point. It follows that the loop must
close up at the beginning of the original border $s$ and hence is homeomorphic to a circle. Since every border
has a successor, these loops partition $L$.

\begin{figure}[h]
\begin{center}
\includegraphics[width=0.5\linewidth]{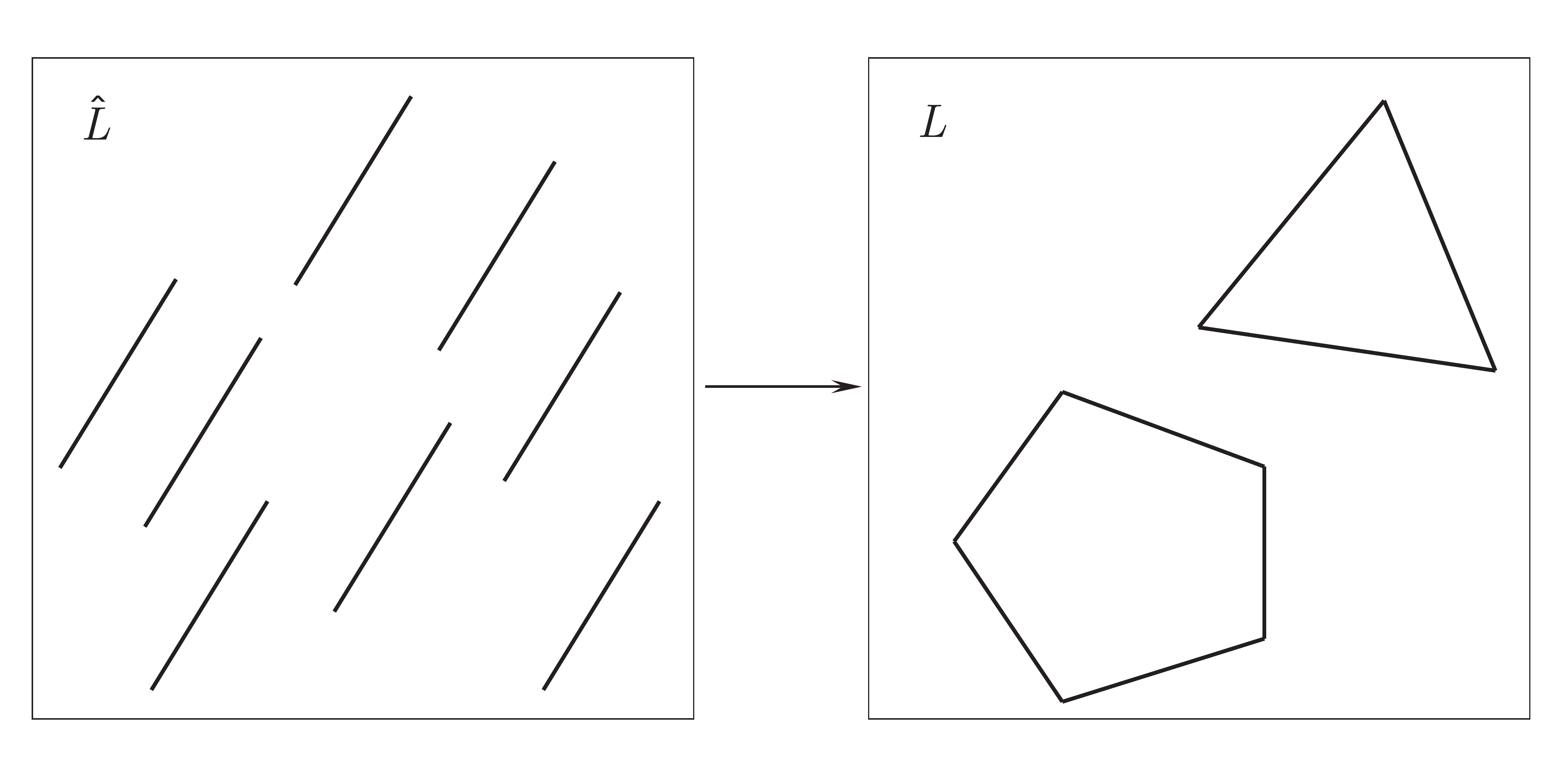}
\caption{The space $L$ is formed from $\hat L$ by connecting each segment to its successor.}
\label{loops}
\end{center}
\end{figure}

For each loop in $L$, associate as its \emph{face}  a closed oriented unit disk, and let $\hat F$ be the disjoint union of all
such faces. Define $F$ to be the quotient space of $L \cup \hat F$ obtained by identifying each loop with the 
boundary of its face (see Figure \ref{loop_fill}), respecting the orientations.

\begin{figure}[h]
\begin{center}
\includegraphics[width=0.5\linewidth]{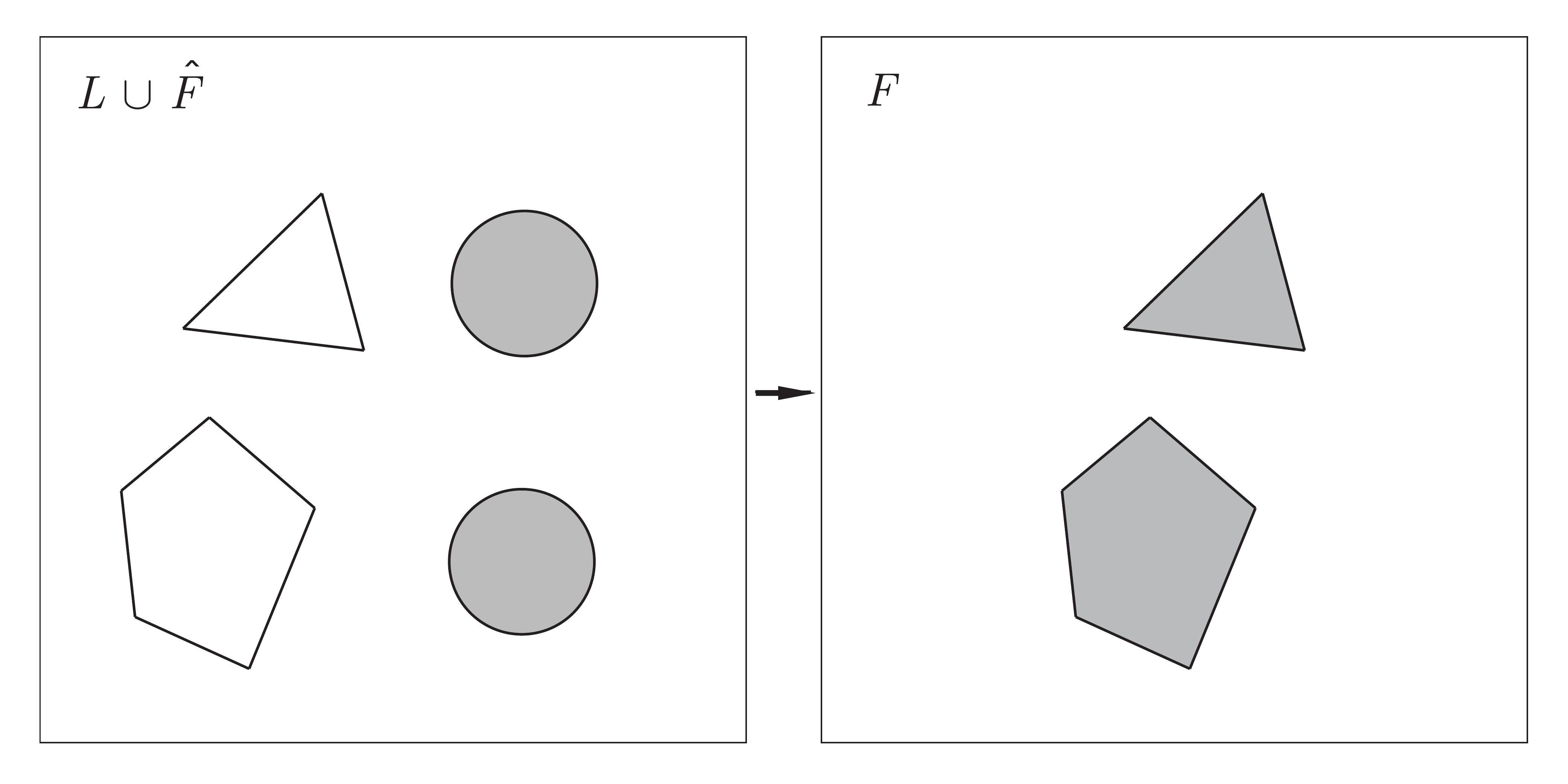}
\caption{$F$ is constructed from $L \cup \hat F$ by ``filling in" each loop with a closed disk.}
\label{loop_fill}
\end{center}
\end{figure}

Finally, we construct the surface $S$ by gluing the disks in $F$ together as follows: Let $S$ be the quotient  of $F$ obtained by identifying each border $[k, k^+]$ with its negative
counterpart $-[k, k^+]$ such that the $0$ endpoint of each border is attached to the $1$ endpoint of its
negative. In essence, we are joining the left and right sides of each edge with opposite orientation. 
Formally, we refer to the images of border pairs under
this quotient map as \emph{edges}. We denote the edge formed by the borders $[k, k^+]$ and $-[k, k^+]$ with $e_k$. 

We now show that $S$ is, in fact, a surface by exhibiting a neighborhood of each point in $S$ that is homeomorphic to an
open disk. By definition, points on the interior of some face have such a neighborhood. Now consider a point $p$ in the interior of
some edge. On each bordering face, the point is contained in a half-disk; in the quotient, these two half-disks meet to 
form an open disk around $p$. Finally, consider a point $p$ at the intersection of edges. In each of the four adjacent
faces, $p$ is contained in a sector of an open disk; in the quotient, these sectors meet along their edges to form an
open disk around $p$. The latter two cases are illustrated in Figure \ref{disks}.

\begin{figure}[h]
\begin{center}
\includegraphics[width=0.4\linewidth]{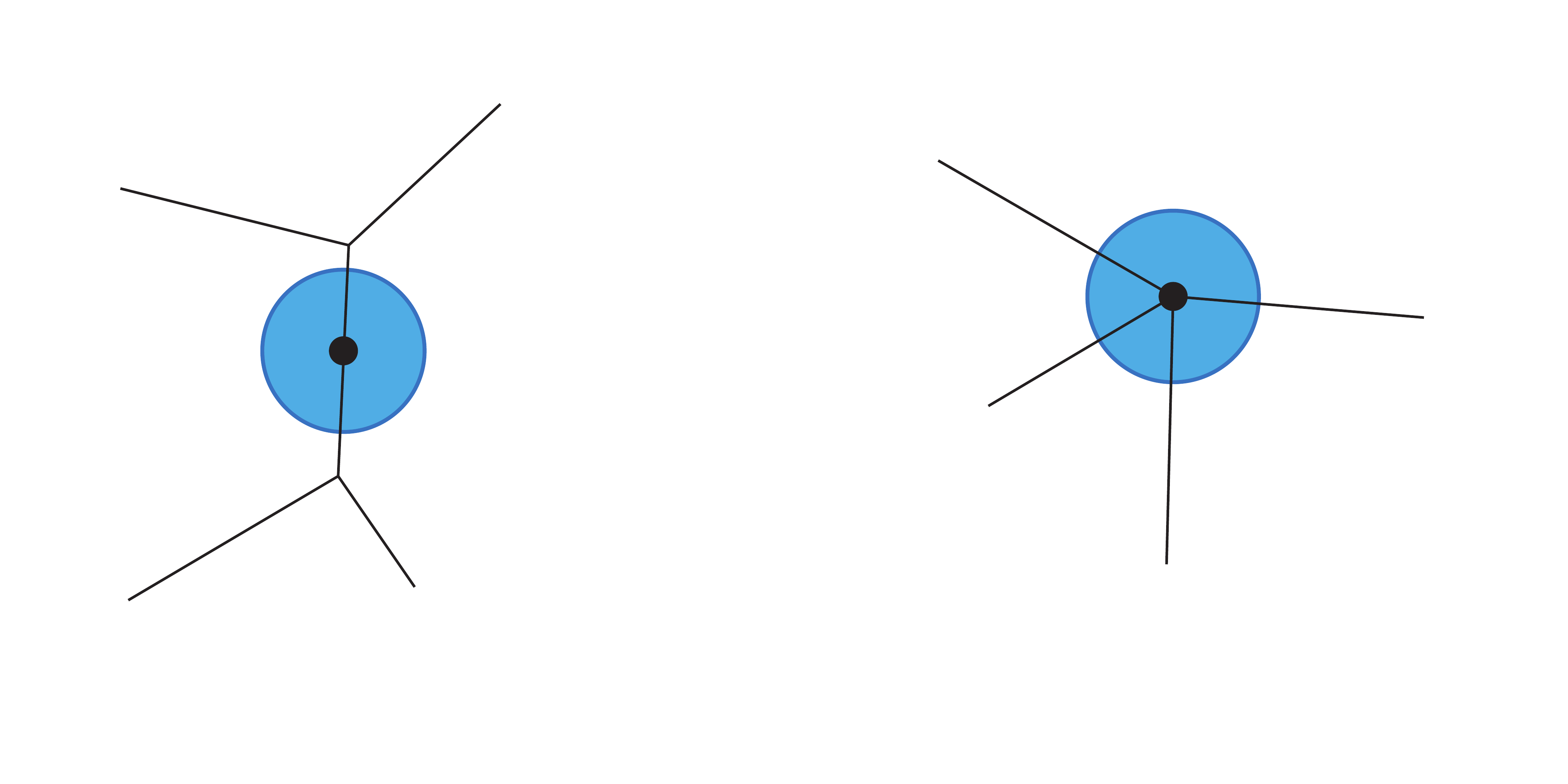}
\caption{The two cases of points on the boundary of some face. In both cases, the neighborhoods of the point
in each face meet to form an open disk in the quotient.}
\label{disks}
\end{center}
\end{figure}

By construction,  $S$ is a closed oriented surface.

Finally, we construct the curve system $\Gamma = \{\gamma_1, \ldots, \gamma_n\}$ on $S$. For each vector $a_i = (a_i^1, \ldots, a_i^{m_i})$ in $A$,
let $\gamma_i$ consist of the edges $e_{a_i^j}$ oriented so that $\gamma_i$ traces these edges in the cyclic order
specified by $a_i$, that is, in the order $e_{a_i^1}$, $e_{a_i^2}$, $e_{a_i^3}$, and so on. If $k \in a_i$, then $-k \in a_j$ for some $i \ne j$, and hence $\gamma_i$ intersects $\gamma_j$; thus $\Gamma$
satisfies condition 1 of Definition \ref{curve_sys}. Condition 2 is guaranteed by the requirement that if $k$ appears in some $a_i$, then
$-k$ is not in $a_i$. Finally, condition 3 follows from the construction of $S$ from closed disks and that the curves in $\Gamma$ comprise
the boundaries of these disks. By construction, $A$ is the encoding of $(S, \Gamma, (k_1, \ldots, k_n))$.

The fact that a curve system allows to reconstruct the surface allows us to associate the topological invariants of the surface to the curves system:

\begin{definition}
We say that a curve system $\Gamma$ has genus $g$ if it defines a  surface $S$ of genus $g$.
\end{definition}

\end{proof}

\begin{example}
Consider the combinatorial curve system
\begin{align*}
a_1 & =  (1,4,3) \\
a_2 & =  (-1,-2) \\
a_3 & = (-4,-5) \\
a_4 & = (2, -3, 5).
\end{align*}
Then  the above proof produces just one disk with 20 borders as shown in Figure \ref{figure:g3example}. The identifications of the borders defines a genus 3 surface.

\begin{figure}[h]
\begin{center}
\includegraphics[width=0.7\linewidth]{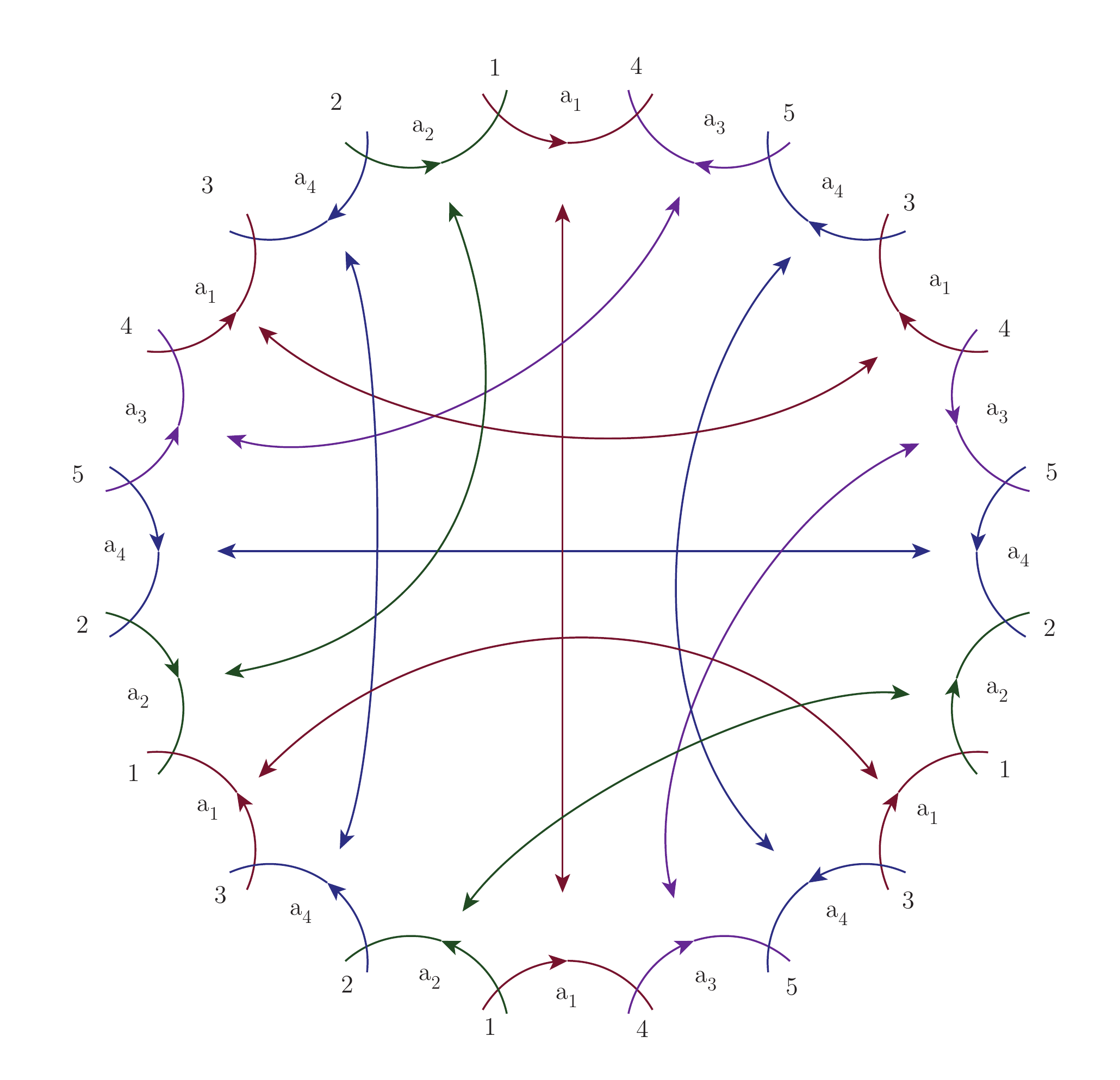}
\caption{Genus 3 surface reconstructed from a curve system}
\label{figure:g3example}
\end{center}
\end{figure}

\end{example}

\section{Periodic Tilings with Parallelograms and Curve Systems}\label{sec:geometric}

In this section, we will introduce a canonical  curve system for a periodic tiling of the plane by marked (or colored) parallelograms.
We assume that this tiling is edge to edge. The set  $\Lambda$ of translations that leave the marked tiling invariant forms a lattice, and the quotient $S=\R^2/\Lambda$ is a torus.

We introduce a curve system on the torus as follows (see Figure \ref{figure:zones}): Pick an arbitrary parallelogram of the tiling, and choose one of its edges. Draw a segment from the midpoint of this edge to the midpoint of the opposite edge. Then keep going, connecting midpoints of edges with midpoint of opposite edges of adjacent parallelograms. The resulting polygonal arc $\gamma$ will eventually close up on $S$. As $\gamma$ intersects only edges parallel to the first edge, the curve $\gamma$ is necessarily simple on $S$. It cannot be null homotopic on $S$: Otherwise it would lift to a  closed curve in $\R^2$, contradicting that $\gamma$ intersects only parallel edges.

\begin{figure}[h]
\begin{center}
\includegraphics[width=0.4\linewidth]{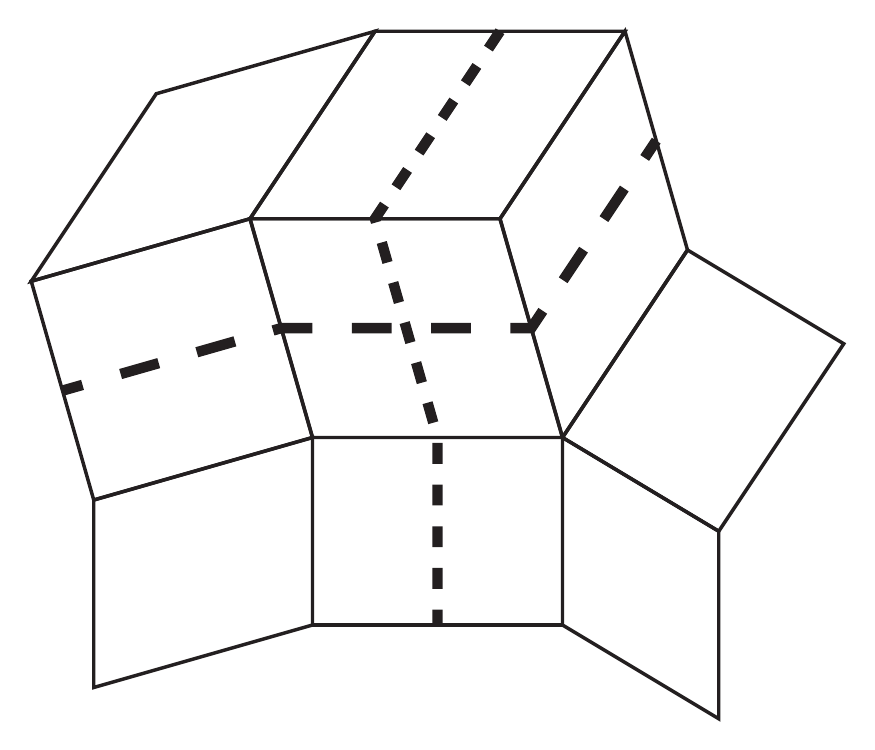}
\caption{A partial zone system for a  tiling with parallelograms}
\label{figure:zones}
\end{center}
\end{figure}

Carrying out this construction for all edges of the parallelogram tiling results in a finite system $\Gamma$ of simple closed curves on $S$.
We choose an arbitrary orientation for each of these curves.

This curve system satisfies the condition of Definition \ref{curve_sys}, and is thus a topological curve system in our sense, which we call the \emph{zone system} of the periodic tiling.

Not all topological curve systems on tori arise this way. We prove:

\begin{theorem}
\label{thm:periodic}
A topological curve system on a torus is the zone system of a periodic tiling if and only if no curve from the curve system bounds a disk, and no two curves from the curve system intersect such that  the segments between two intersections bound a disk. We call a curve system with these properties an \emph{essential} curve system.
\end{theorem}

We divide the proof into a few simple lemmas:

Let $\Gamma$ be the zone system of a periodic tiling.
Every curve from the zone system  lifts to a quasigeodesic in $\R^2$, i.e. stays at bounded distance from a line. Moreover, the direction of this line is uniquely determined. We can also find this line by looking at the homotopy  class
of the curve on the torus,  representing it by a closed geodesic, and lifting it to $\R^2$.
For a curve $\gamma$ from the curve system, we denote by $T_\gamma$ the  tangent vector of this line of the same length as the length of the closed geodesic. The orientation of this vector, which we call the \emph{zone vector} of $\gamma$, is determined, as we have chosen orientations for all zone curves.

\begin{lemma}\label{lemma:zonevec} The zone vectors $T_\gamma$ of a periodic tiling are already determined by the period lattice and the encoding of the zone system of the tiling.
\end{lemma}

\begin{proof}
Pick a basis $A$, $B$ of the period lattice of the periodic tiling, and denote the corresponding homology classes on the quotient torus by $\alpha$ and $\beta$. Then $\alpha$ and $\beta$ form a homology basis of the torus, and each zone curve $\gamma$ is homologous to an integral linear combination $a_\gamma\alpha+b_\gamma\beta$. Then $T_\gamma=a_\gamma A+b_\gamma B$ is the zone vector.
\end{proof}

\begin{lemma}
For any zone curve $\gamma$, we have that the inner product $\gamma'(t) \cdot T_\gamma >0$. In particular, if two zone curves intersect multiple times, they do so everywhere with the same sign. 
\end{lemma}
\begin{proof}
To see this, note that the zone curve $\gamma$ is defined by intersecting edges that are all parallel to each other. If we assume without loss of generality, that these edges are all horizontal, then $\gamma$ would always point upward (say). The same must then hold for the quasigeodesic, and hence for $T_\gamma$.  At an intersection of two zone curves, the intersecting segments within the parallelogram are parallel to the edges. If the sign at two such intersections were different, two such segments belonging to the same zone curve would have opposite orientation, which is impossible as the sign of $\gamma'(t) \cdot T_\gamma$ has to remain the same.
\end{proof}

The next Lemma proves one direction of  Theorem \ref{thm:periodic}:

\begin{lemma} No curve from the curve system bounds a disk, and no two curves from the curve system intersect such that  the segments between two intersections bound a disk.
\end{lemma}

\begin{proof}
Suppose the contrary, and consider the lifted curve(s) in $\R^2$. Because disks lift to disks, segments of these curves would still bound disks in $\R^2$. For the sake of concreteness, assume that the first zone curve intersects vertical edges from the left to the right. As it proceeds monotonically to the right, it can never be closed. This proves the first claim.  For the second claim, we would obtain two consecutive intersections of the two zone curves with opposite sign, which is impossible by the previous Lemma.
\end{proof}

\begin{lemma}\label{lemma:nondegenerate}
For any parallelogram in a periodic tiling, the two zone curves from $\Gamma$ associated to the pairs of opposite edges of that parallelogram form a basis
of the rational homology of $T$. In other words, the corresponding zone vectors are linearly independent.
\end{lemma}
\begin{proof}
We claim that the two curves have non-zero intersection number. First, they do intersect in the given parallelogram. If the intersection number was zero, there would be another intersection of the two curves with opposite sign, which is impossible. 
\end{proof}

To prove the other direction of Theorem \ref{thm:periodic}, we need to construct a periodic tiling with a given zone system. This will involve the assignment of geometric data to the tiling, which we will now describe.

We  assign geometric data to the zone curves of a periodic tiling as follows: Denote by $V_\gamma$ the edge vector of the parallelograms that are being intersected by the zone curve $\gamma$. We choose the orientation of $V_\gamma$ such that $\det(T_\gamma,V_\gamma)>0$. The set of vectors $V_\gamma$ is called the \emph{geometric data} of the tiling.

The edge vectors and zone vectors are related by a simple compatibility condition:

\begin{lemma}\label{lemma:compatability}
Let $\alpha$ and $\beta$ be two zone curves that intersect in a parallelogram. Then $\det(T_\alpha, T_\beta)$ and $\det(V_\alpha, V_\beta)$ have the same sign. Moreover, this sign is already determined by the encoding of the curve system and the orientation of the torus.
\end{lemma}

\begin{proof} Assume without loss of generality that $\det(T_\alpha, T_\beta)>0$. Then $V_\alpha$ and $V_\beta$ span a positively oriented parallelogram, hence their determinant must be also positive. The zone vectors are determined by the encoding and a choice of an oriented homology basis with associated basis of the period lattice. Any other basis will differ by an affine transformation with positive determinant, thus leading to another set of zone vectors that are transformed using the same affine transformation. This clearly  doesn't affect the signs of the determinants $\det(T_\alpha, T_\beta)$. In fact, the sign of the determinant is the same as the sign of the intersection number of the zone curves.
\end{proof}

This condition is necessary for a set of  data $V_\gamma$ to be the geometric data of a tiling. It turns out that this condition is also sufficient:

\begin{theorem}\label{thm:compatability}
Given an essential curve system $\Gamma$ on an oriented  torus, and a
 set of edge vectors $V_\gamma$ assigned to all $\gamma\in\Gamma$ such that for any pair of intersecting curves $\alpha, \beta\in\Gamma$,   $\det(V_\alpha, V_\beta)$ has the same sign as the intersection number of $\alpha$ and $\beta$, there is a tiling  of the plane by parallelograms such that its zone system is $\Gamma$.
\end{theorem}
\begin{proof}

To construct the tiling, we choose for every intersection point of any two curves of the curve system a parallelogram with edge vectors $V_\alpha$ and $V_\beta$. Because the curve system is essential, every such intersection gives a determinant condition. We identify edges of two parallelograms if the corresponding intersections are connected by a segment from one of the curves of the curve system. After gluing the parallelograms together, we obtain a cone metric on the torus, where the cone points correspond to the disks into which the curve system divides the torus. We have to show that the cone angles at each cone point are $2\pi$. Choose a vertex $v$, and consider all edges emanating from that vertex in counterclockwise order. 
This order can be obtained by following the segments of the curve system around the vertex, switching to another curve at each intersection, just as in the proof of Theorem \ref{bigone}.

By the determinant condition, the counterclockwise angle from one edge to the next of the same parallelogram is  positive and less than $\pi$. Thus, the total cone angle $\phi_v$ will be a positive integral multiple of $2\pi$. By the Gauss-Bonnet Theorem, the  sum $\sum_v (2\pi-\phi_v) = 0$. Thus,   $\phi_v=2\pi$ for all vertices $v$.
\end{proof}

We now show that for an essential curve system, the moduli space of periodic tilings is nonempty:

\begin{theorem}
For a given essential curve system, there are always geometric data satisfying the determinant condition in Theorem \ref{thm:compatability}.
\end{theorem}

\begin{proof}
For the given torus, pick a homology basis and a basis of $\R^2$. This allows one to define the zone vectors $T_\gamma$.  Now define geometric data $V_\gamma = R\cdot T_\gamma$ where $R$ is any  rotation matrix. These data obviously satisfy the determinant condition. Note that  in general, the zone vectors of the corresponding tiling will be different from the chosen vectors $T_\gamma$, but this is irrelevant, as the sign of the determinants depend only on the signs of the intersection numbers of the zone curves.
\end{proof}

\section{The  Structure Theorem for Periodic Tilings by Parallelograms}

\subsection{Notation and Main Result}

Our first goal is to set up a moduli space and a Teichmuller space of marked periodic tilings of the plane by parallelograms.

We first define the the Teichmuller space: 
Let $\tau$ be a periodic tiling of the plane by (labeled) parallelograms, where we identify tilings that just differ by a translation.
Denote by $\Lambda=\Lambda(\tau)$  the period lattice of the tiling and by $S=\C/\Lambda$  the quotient torus. Choose a basis $a, b$ of the period lattice $\Lambda$ --- this choice is equivalent to a choice of a homology basis $\alpha, \beta$ of $S$. The latter allows us to identify $S$ with a fixed torus $S_0$ (say the square torus) such that $\alpha$ and $\beta$ are identified with the standard basis 
$\alpha_0 =\begin{pmatrix}1\\ 0\end{pmatrix}$ and $\beta_0 =\begin{pmatrix}0\\ 1\end{pmatrix}$ of $S_0$. This identification is unique up to isotopy. We call a periodic tiling $\tau$ together with a choice $a, b$ of a basis of its period lattice a \emph{marked} tiling.

The procedure from section \ref{sec:geometric} then defines  a curve system $\Gamma$ on $S$ and thus on $S_0$, unique up to relabeling, cyclic permutations, and choices of orientation as explained.

Denote the set of all marked tilings with curve system $\Gamma$ by $\tilde \cM(\Gamma)$.

For a fixed curve system $\Gamma$ on $S_0$, any choice of edge data satisfying the compatibility conditions allows us to construct a periodic tiling, unique up to translations, together with a basis of the period lattice. This reduces the description of  $\tilde \cM(\Gamma)$ to the description of a subset of $\C^n$ that is characterized by a set of compatibility conditions.

The group $GL(2,\R)$ acts on $\tilde \cM(\Gamma)$ by left multiplication on the vertices of the tiling (and the basis vectors of the marking). As  every basis $(a, b)$ can be uniquely mapped to the standard basis of $\R^2$, we see that 
\[
\tilde \cM(\Gamma) = \tilde \cM_0(\Gamma) \times GL(2,\R)\ ,
\]
where $\tilde \cM_0(\Gamma)$ denotes the set of all elements of $\tilde \cM(\Gamma,\alpha,\beta)$ where $a =\begin{pmatrix}1\\ 0\end{pmatrix}$ and $b =\begin{pmatrix}0\\ 1\end{pmatrix}$.

Finally, the group $SL(2,\Z)$ acts on $\tilde \cM(\Gamma)$ and $\tilde \cM_0(\Gamma)$ by changing the basis of the lattice. The quotient spaces
\[
 \cM(\Gamma) = \tilde \cM(\Gamma)/SL(2,\Z) \qquad\hbox{and}\qquad  \cM_0(\Gamma) = \tilde \cM_0(\Gamma)/SL(2,\Z)
\]
are called the moduli spaces of periodic tiling with underlying curve system $S$.

Thus $\tilde \cM(\Gamma)$ serves us as a sort of Teichmuller space for periodic, with the caveat that this space is not simply connected, due to the presence of  rotations in $SL(2,\R)$. However, $\tilde \cM_0(\Gamma)$ is simply connected. More precisely, we have the following structure theorem:

\begin{theorem}\label{theorem:structure}
For a given curve system $\Gamma$ of genus 1 consisting of $n$ curves, the set $\tilde \cM(\Gamma)$ is naturally a non-empty  open subset  of $\C^{n}$. Its boundary is stratified by pieces of hypersurfaces given by equations of the form 
$\{ e\in\C^n:  \det_{\R}(e_i, e_j) = \im (\overline{e_i}e_j)=0\}$.

 Moreover, $\tilde \cM_0(\Gamma)$ is star shaped with respect to a distinguished point in $\tilde \cM_0(\Gamma)$. In particular, $\tilde \cM(\Gamma)$ is homotopy equivalent to $S^1$.
\end{theorem}

\begin{example}
Consider the curve system $\Gamma$ with combinatorial data $(1), (-1)$. This consists of just two zone curves that intersect in a single parallelogram. In this case $\tilde \cM_0(\Gamma)$ consists of a single point represented by the square tiling.
\end{example}

\begin{example}\label{example:proof1}
To illustrate the proof, we will use the following example of a curve system as we go along:
\begin{align*}
a_1 & =  (1,2,3,4) \\
a_2 & =  (-1,5) \\
a_3 & = (-3,6) \\
a_4 & = (-2,-5,-4,-6).
\end{align*}
\begin{figure}[h]
\begin{center}
\includegraphics[width=0.7\linewidth]{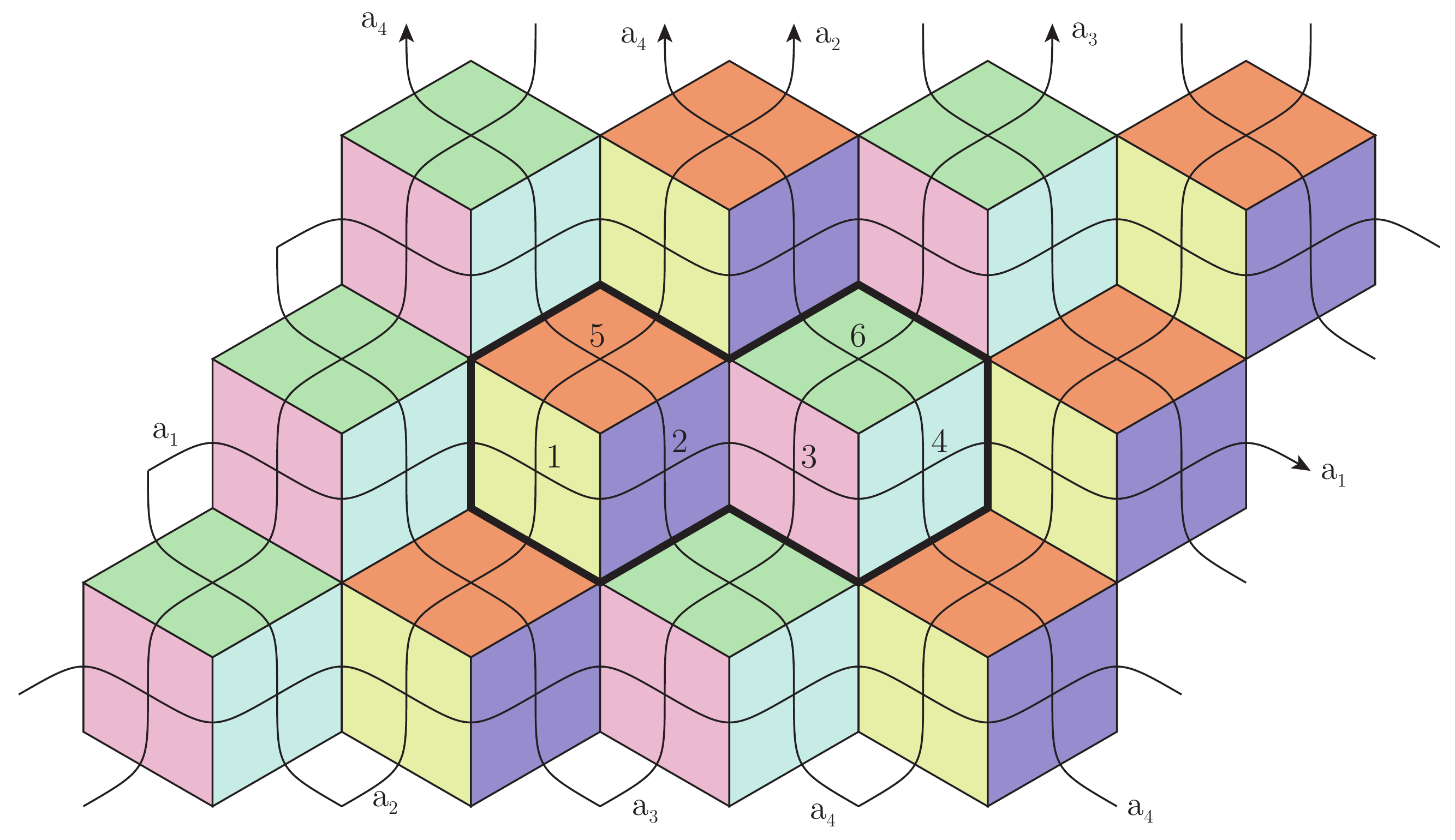}
\caption{Periodic Tiling by six colored parallelograms, with curve system and fundamental domain.}
\label{figure:sixrhombs}
\end{center}
\end{figure}
\end{example}

The proof of the theorem will be carried out in the subsequent subsections. In subsection \ref{subsec:data} we identify $\tilde \cM(\Gamma)$ with an open subset of $\C^n$, using edge vectors as parameters.
Our next goal is to find an explicit  distinguished point in $\cM(\Gamma)$. Its  edge vectors are found as entries of an eigenvector $e^0$ of the generalized intersection matrix, which will be introduced in Subsection \ref{subsec:intersect}. 
Using geometric arguments we finally show that the convex combinations with $e^0$ of any other point in $\cM(\Gamma)$  that has the same periods as $e^0$ still lies in $\cM(\Gamma)$.

\subsection{The Generalized Intersection Matrix}\label{subsec:intersect}
We will now introduce our main tool for proving the structure theorem. We will give the definitions and basic properties  for curve systems of arbitrary genus,
and specialize later. 

Let $\Gamma$ be a curve system of genus $g$ consisting of zone curves $\gamma_i$ for $i=1,\ldots, n$. Recall that the curves $\gamma_i$ come with a natural orientation, and that the underlying Riemann surface constructed form the combinatorial data also has a natural orientation.

\begin{definition}
Given a curve system $\gamma_i$ for $i=1,\ldots, n$, we define the {\em generalized intersection matrix} as
\[
C = (c_{i,j}) =\gamma_i \cdot \gamma_j
\]
where  $\gamma_i \cdot \gamma_j$ is the intersection number of the two cycles $\gamma_i$ and $\gamma_j$, counting
multiplicity and taking orientation into account.
\end{definition}

The combinatorial data of a curve system $\Gamma$ can be used to easily compute the intersection matrix $C$.

\begin{definition}
A curve system $\gamma_i$ is called  {\em non-degenerate} if it generates the rational homology.
\end{definition}

\begin{lemma}
For a non-degenerate curve system $\Gamma$ of genus $g$, the generalized intersection matrix $C$ has rank $2g$.
\end{lemma}
\begin{proof}
If the curve system happens to be a homology basis, the generalized intersection matrix is just the usual intersection matrix,
which is well known to be non-degenerate, and thus of rank $2g$. In general, the cycles from the curve system are a linear combination of the cycles from a homology 
basis, and therefore the generalized intersection matrix has rank at most $2g$. Similarly, as the curve system is non-degenerate, the cycles from a homology basis can be written as linear combination of the cycles from the curve system, so that the rank of the usual intersection matrix is at most the rank of the generalized intersection matrix.
\end{proof}

\begin{example}
The curve system given by $(1, 4, 3)$, $(-1, -2)$,
$ (-4, -5)$ and
$(2, -3, 5)$ has genus 3, but the curves do not generate the homology. The generalized intersection matrix
\[
C=\begin{pmatrix}
0 & 1 & 1 & 1 \\
-1& 0 & 0 & -1\\
-1 & 0 & 0 & -1\\
-1 & 1 & 1 & 0\\
\end{pmatrix}
\]
has rank 2.
\end{example}

\begin{example}\label{example:proof2}
The curve system from Example \ref{example:proof1} has
 generalized intersection matrix
\[
C=\begin{pmatrix}
0 & 1 & 1 & 2 \\
-1& 0 & 0 & 1\\
-1 & 0 & 0 & 1\\
-2 & -1 & -1 & 0\\
\end{pmatrix}
\]
has rank 2.
\end{example}

However, by Lemma \ref{lemma:nondegenerate}, a curve system of genus 1 is always non-degenerate.

In subsection \ref{subsec:canonical} we will need information about the eigenvalues of $C$ to construct a tiling with canonical geometrical data:

\begin{lemma}\label{lemma:eigen}
For a non-degenerate curve system $\Gamma$ of genus $g$, the generalized intersection matrix $C$ has purely imaginary non-zero eigenvalues that come in $g$ conjugate pairs.
\end{lemma}
\begin{proof}
As a skew symmetric matrix, $C$ is diagonalizable over the complex numbers with purely imaginary eigenvalues, coming in conjugate pairs.
As the curve system is non-degenerate, there will be precisely $g$ such pairs.
\end{proof}

For a given choice of a homology basis of the  surface $S$, we can express the curves $\gamma_i$ in terms of the homology basis and thus the generalized intersection matrix in terms of the intersection numbers of the homology basis:

Let $\alpha_i$ and $\beta_i$, $i=1,\ldots,n$ be a canonical homology basis of $S$, i.e one with intersection numbers $\alpha_i \cdot \alpha_j=0=\beta_i\cdot \beta_j$ and
$\alpha_i\cdot \beta_j =\delta_{i,j}$. 

Then there are integers $a_{ij}$ and $b_{ij}$ such that
\[
\gamma_i = \sum_{j=1}^g a_{ij}\alpha_j + b_{ij}\beta_j
\]

Thus

\begin{align*} 
   \gamma_i \cdot \gamma_j = {}& \left(\sum_{k=1}^g a_{ik}\alpha_k + b_{ik}\beta_k\right) \cdot \left(\sum_{k=1}^g a_{jk}\alpha_k + b_{jk}\beta_k\right)\\
 = {}& \sum_{k=1}^g a_{ik}b_{jk} - b_{ik}a_{jk}  
\end{align*}

If we let $A=(a_{ij})$ and $B=(b_{ij})$ be the corresponding $n\times g$ matrices,  we obtain
\begin{lemma}\label{lemma:intersect}
\[
C = A B^t - B A^t 
\]
\end{lemma}

Note that in the case that $g=1$, $A$ and $B$ are just vectors.

One may wonder whether the generalized intersection matrix contains enough information to reconstruct the combinatorial curve system. This is, however, not the case:
Suppose we have a periodic tiling where three parallelograms fit together to form a hexagon. Then the subdivision of the hexagon by the parallelograms can be switched to another subdivision, as shown in  Figure \ref{figure:rototiler}. This switch changes the tiling locally and subjects the combinatorial curve system to a certain permutation, but leaves the generalized intersection matrix unchanged. This operation on finite rhombic tilings was first introduced by Alan Schoen \cite{schoen0}  around 1980 in a computer game called Rototiler.

\begin{figure}[h]
\begin{center}
\includegraphics[width=0.7\linewidth]{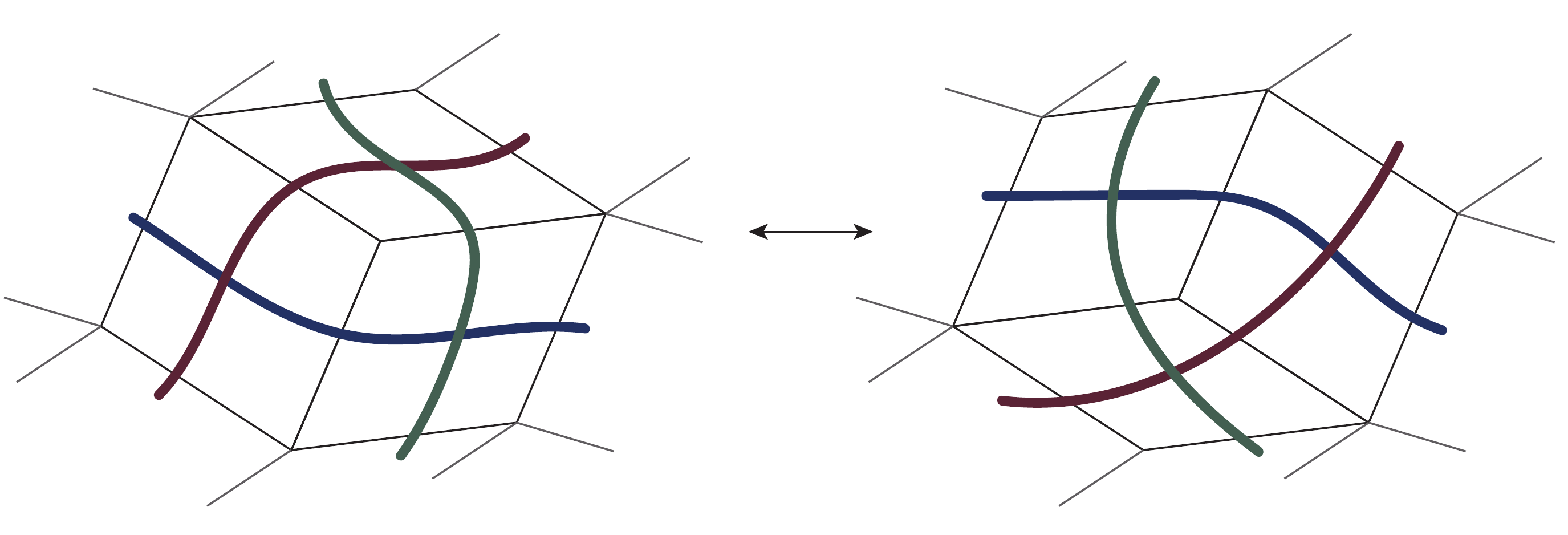}
\caption{A Rototiler move on hexagons.}
\label{figure:rototiler}
\end{center}
\end{figure}

\subsection{Edge Data and Zone Vectors}\label{subsec:data}

 We now add geometric data to a curve system, restricting all attention to $g=1$ as in section \ref{sec:geometric}. We will reformulate the compatibility condition in Lemma \ref{lemma:compatability} and relate it to the intersection matrix and the period lattice. 
 
 Recall that for a given tiling by parallelograms,  the zone system was the system of curves
 that traverse parallelograms across opposite edges. Hence we can assign the edge vector $e_i=V_{\gamma_i} \in\C$ of the traversed edge to each zone curve $\gamma_i$. As in section \ref{sec:geometric}, we choose the orientation of that edge so that  $\gamma_i'$ and $e_i$ form an oriented basis of $\R^2=\C$.

\begin{definition}
We call the vector $e=(e_i)_{i=1}^n$ the edge data of the tiling. The components $e_i\in\C$ are the edge vectors of the parallelograms. 
\end{definition}
 
 The compatibility condition for edge data can be expressed as follows:
 
  \begin{definition}
 A set of edge vectors $e_i$ is {\em admissible} if for every pair $(i,j)$ where the zones $i$ and $j$ intersect, the edge vectors $e_i$ and $e_j$ are compatible. This is the case if and only if $ c_{ij}\det_{\R}(e_i, e_j)>0$ holds for all $i,j$ where $c_{ij}\ne0$.
 \end{definition}

 As the admissibility condition is open, we have:
\begin{corollary}\label{cor:localdim}
For any periodic tiling with a curve system $(\gamma_1,\ldots,\gamma_n)$ of $n$ curves and (admissible) edge data $e\in \C^n$, there is a neighborhood of $e$ in $\C^n$ of admissible edge data, and hence a periodic tiling with these edge data.
\end{corollary}
This proves the dimension claim in the Structure Theorem \ref{theorem:structure}. 
 
 Moreover, a boundary point $e$ of the set of admissible edge data  satisfies $\det_{\R}(e_r, e_s)=0$  for a pair of indices where $c_{rs}\ne0$. Thus

\begin{proposition}
 For a pair of indices  $r,s$ where $c_{rs}\ne0$,  let $H_{r,s}$ be the hypersurface given by $\det_{\R}(e_r, e_s)=0$. Denote by $B_{r,s}$ the possibly empty subset of $H_{r,s}$ 
 given by the relaxed admissibility conditions $\det_{\R}(e_i, e_j) c_{ij}>0$ for all $\{i,j\}\ne  \{r,s\}$ where $c_{ij}\ne0$.
  Then the boundary of the set of admissible edge data consists of  the union of all sets $B_{r,s}$.
 \end{proposition}
 
 This proves the claim about the local nature of the boundary of $\tilde M (\Gamma)$ in the Structure Theorem. Observe that the hypersurface $\det_{\R}(e_r, e_s)=0$ equations are non-convex.

 If we follow a zone curve $\gamma_i$ once around on the quotient torus $S$, it develops in $\C$ to a zone vector $z_i:=T_{\gamma_i}$ (compare Lemma \ref{lemma:zonevec}) in the period lattice of the tiling, which we can determine explicitly: 
  
 \begin{lemma}\label{lemma:develop}
 Let $e = (e_i)$ and $z = (z_i)$ be the column vectors of the complex edge and zone vectors. Then
 \[
 z = C e
 \]
 \end{lemma}
 \begin{proof}
 Consider the parallelograms that are being traversed by a zone. The zone path is homotopic to the edges of these parallelograms that are \emph{not} being intersected by the zone path. Adding them all up with the right sign proves the claim. 
 \end{proof}
 
 Simplifying the notation from the previous subsection to the case $g=1$,  take a canonical homology basis $\alpha$, $\beta$ of the torus $S$, and write
 \[
 \gamma_i = a_i \alpha + b_i \beta\ .
 \]
 Recall that the coefficient vectors $A=(a_i)$ and $B=(b_i)$ are column vectors in $\Z^n$ so that by Lemma \ref{lemma:intersect}
 \[
 C = A B^t- B A^t \ .
 \]
 Note that $A$ and $B$ are linearly independent. Otherwise, the curve system $\gamma_i$ would be degenerate.
 
 Now suppose that $\alpha$ and $\beta$ are developed to complex numbers $a$ and $b$ --- these will form the  basis of the lattice.
 Then $\gamma_i$ is developed to $z_i = a_i a+ b_i b$, as $\gamma_i$ is homologous to $a_i \alpha +b_i \beta$. Combining this with  Lemma \ref{lemma:develop} gives
 \begin{lemma}
 \[
 a A+ bB = C e = (A B^t- B A^t) e
 \]
 \end{lemma}
 This lemma allows to determine a basis for the period lattice of a tiling if we are given a curve system, edge data, and a choice of a homology basis for the torus.
 
 \begin{corollary}\label{cor:linear}
 Two (admissible) sets $e$ and $e'$ of edge data determine the same lattice if $Ce=Ce'$.
 \end{corollary}
 The conclusion is actually a bit stronger --- for a given homology basis, both edge data produce the same basis for the lattice. It might well be that
 $Ce\ne Ce'$ but the generated lattices are still the same. The important point here is that staying in the same lattice is just a linear condition on the edge data. As the rank of $C$ is 2, we obtain:
 
  \begin{corollary} For every admissible edge data $e\in \C^n$ there is a neighborhood $U$ of $e$ in $\C^n$ and an affine subspace $E_e$ of complex codimension 2 through $e$ in $\C^n$ such that all edge data in $U\cap E_e$ are admissible and define tilings with the same 
 basis of the period lattice.
 \end{corollary}

\subsection{Canonical Edge Data}\label{subsec:canonical}

 We have seen in section \ref{sec:geometric} that an essential curve system of genus one always has admissible edge data and thus comes from a periodic tiling.
 Our next goal is to show that  such   edge data can be defined quite canonically, thus leading to a canonical tiling with the given curve system. This will be our distinguished tiling mentioned in the Structure Theorem \ref{theorem:structure}.
 
 Recall from Lemma \ref{lemma:eigen} that the generalized intersection matrix $C$ of a curve system $\Gamma$ of genus 1 has rank 2 with a pair of conjugate imaginary eigenvalues.
 \begin{definition}
 We call   a non-zero eigenvector $e^0$ of $C$ with positive imaginary  eigenvalue  the canonical edge data. It is uniquely determined up to multiplication by a complex number.
 \end{definition}

 Then we claim:
 
 \begin{theorem}
 The canonical edge data is admissible.
 \end{theorem}
 
 \begin{proof}
  For  canonical edge data $e^0$, the zone vector $z$ is given by Lemma \ref{lemma:develop} as 
 \[
 z = Ce^0 = \lambda \sqrt{-1} e^0
 \]
  for some $\lambda>0$. Thus, if two zone curves $\gamma_i$ and $\gamma_j$ intersect, then $\det_{\R}(z_i,z_j)$ and $\det_{\R}(e^0_i,e^0_j)$
  have the same sign.
 \end{proof}

\begin{corollary}
The space $\tilde \cM(\Gamma)$ is non-empty.
\end{corollary}

 \begin{example}\label{example:proof3}
 The intersection matrix $C$ from Example \ref{example:proof2} has eigenvalues $\pm 2\sqrt{-2}$ and $0$, and the eigenvector for $2\sqrt{-2}$ is 
 \[
 e^0 = \left(
 -1-2\sqrt{-2}, 1-\sqrt{-2},1-\sqrt{-2},3
 \right)
 \]
 In Figure \ref{fig:canonical} we show the original ``hexagonal'' fundamental domain and the canonical one using the edge  data $e^0$.
 
 \begin{figure}[h] 
   \centering
   \includegraphics[width=.5\linewidth]{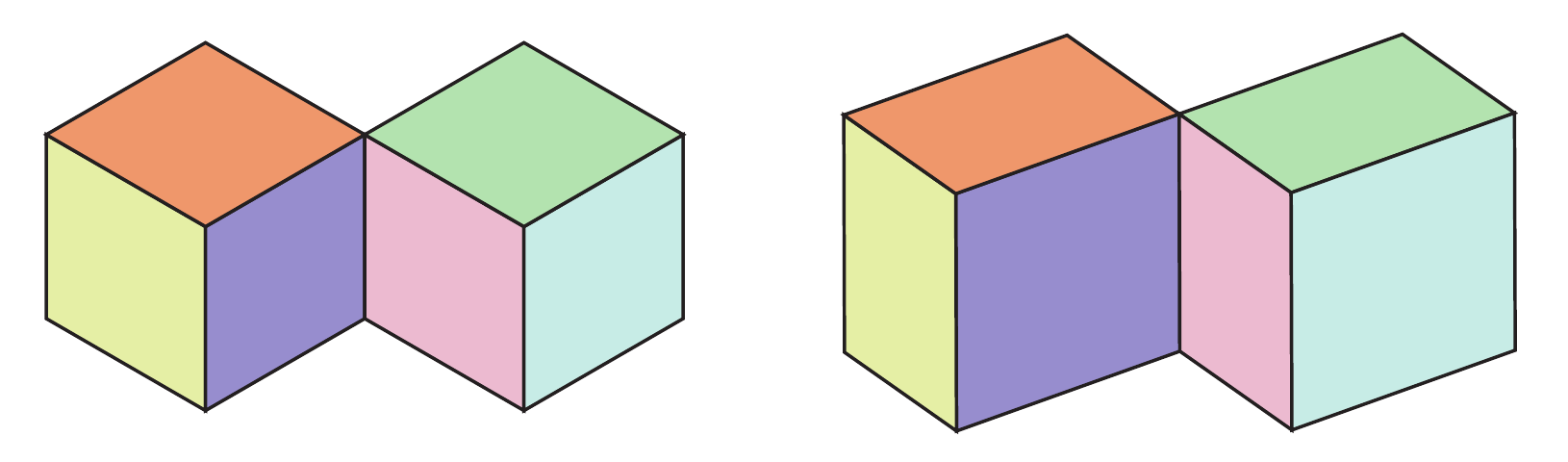} 
   \caption{Fundamental domains for original and canonical edge data}
   \label{fig:canonical}
\end{figure} 

 \end{example}

 \begin{question}
Is the canonical tiling for a given curve system  optimal in some sense, or can it be characterized geometrically?
 \end{question}
 
Finally, we will prove that  
 
 \begin{theorem} Let  $\Gamma$ be a curve system of genus 1, and $(\alpha, \beta)$ a canonical basis for the homology of $S$.
 Let $e^0$ be the canonical edge data, and $e^1$ be another admissible edge data vector with $C(e_1-e_0)=0$. Then $e^t=(1-t)e_0+t e_1$ are all admissible.
  \end{theorem}
 
 \begin{proof}
 By Corollary \ref{cor:linear} we know that the two cycles $\alpha$ and  $\beta$ are mapped to the same vectors $a$ and $b$ for all edge data $e^t$. 
  Denote the  edges  corresponding to zone $\gamma_i$ for parameter value $t$ by $e^t_i$. Then   the zone vectors by $z=z^t= C e^t$ are independent of $t$ by assumption.
 
 We have to verify that the admissibility condition holds. To this end, let $\gamma_i$ intersect $\gamma_j$ positively. 
 By applying a linear transformation if necessary, we can assume for simplicity that for a given fixed intersecting pair, the two zone vectors $z_i$ and $z_j$ are the coordinate vectors $1$ and $\sqrt{-1}$.  For $t=0$, the edge data $e^0$ is an eigenvector of the generalized intersection matrix $C$, i.e. $Ce^0=\lambda\sqrt{-1}e^0$ for some $\lambda>0$. As $z=C e^0$, we get that $e^0 =-\frac1\lambda \sqrt{-1} z$. 
 
Now observe that $e^1_i$ lies in the half plane $\im z<0$ while $e^1_j$ lies in the half plane $\re z>0$ because the edge data is admissible for $t=1$. Moreover, $\det_{\R} (e^1_i,e^1_j)>0$.
We have to show that $\det_{\R} (e^t_i,e^t_j)>0$ for all $t\in[0,1]$. As $e^t_k$ is a convex combination of $e^0_k$ and $e^1_k$ for all $k$, $e^t_i$ stays in the half plane $\im z<0$ and $e^t_j$ stays in the half plane $\re z>0$. We now distinguish four cases, depending on the location of  $e^1_i$ and $e^1_j$ in the quadrants: 

If both $e^1_i$ and $e^1_j$  lie in the quadrant $\re z>0, \im z<0$, the deformation to $e^0_i$ and $e^0_j$ increases the angle between $e^t_i$ and $e^t_j$ with decreasing $t$, so that $\det_{\R} (e^t_i,e^t_j)>0$. 

If  $e^1_i$  lies in the quadrant $\re z>0, \im z<0$ and $e^1_j$  lies in the quadrant $\re z>0, \im z>0$, the same holds for all $e^t_i$ and $e^t_j$, again preserving $\det_{\R} (e^t_i,e^t_j)>0$.

If  $e^1_i$  lies in the quadrant $\re z<0, \im z<0$ and $e^1_j$  lies in the quadrant $\re z>0, \im z<0$, the same holds for all $e^t_i$ and $e^t_j$, again preserving $\det_{\R} (e^t_i,e^t_j)>0$.
If  $e^1_i$  lies in the quadrant $\re z<0, \im z<0$ and $e^1_j$  lies in the quadrant $\re z>0, \im z>0$,  the same holds for all $e^t_i$ and $e^t_j$, and  the deformation to $e^0_i$ and $e^0_j$ decreases the angle between $e^t_i$ and $e^t_j$ with decreasing $t$. Thus  $\det_{\R} (e^t_i,e^t_j)>0$ remains valid again.

\begin{figure}[h] 
   \centering
   \includegraphics[width=5in]{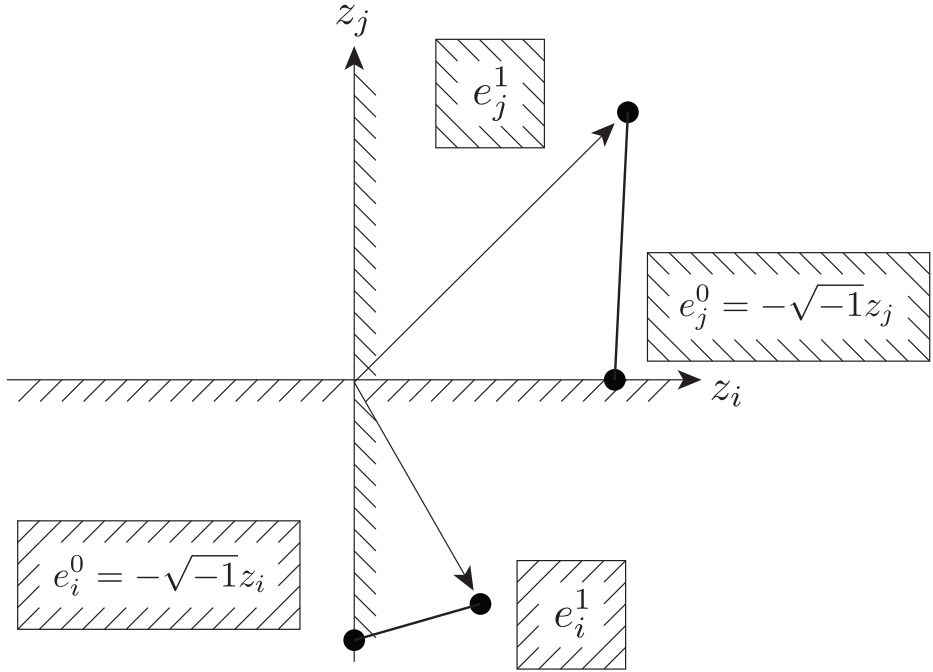} 
   \caption{Convex combination of edge data}
   \label{fig:example}
\end{figure}

Thus in all cases the convex combinations $e^t_i$ and $e^t_j$ satisfy the compatibility condition, and the edge data $e^t$ are admissible.
 \end{proof}
 
\subsection{Epilogue}

As $\cM(\Gamma)$ is naturally a subset of $\C^n$ via the edge data $e$, and two points $e$ and  $e'$ represent similar tilings if they are projectively equivalent, it is natural to consider
the quotient space $\cP(\Gamma)$ as a subset of complex projective space.  Even better, the area $\area(e)$ of a fundamental domain of the tiling defines a sesquilinear form on $\C^n$ that is   invariant under rotations, so that one could identify 
\[
\cP(\Gamma) = \{e \in \cM(\Gamma): \area(e)=1\}/S^1
\]
in the spirit of \cite{Thurston1} and \cite{BG1} to put a geometric structure on $\cP(\Gamma)$. But alas, the area form is represented on $\C^n$ by $\sqrt{-1}C$, where $C$ is the generalized interesction matrix of $\Gamma$. Its rank is 2, leaving us with a highly degenerate geometry on $\cP(\Gamma)$.

 \bigskip

\bibliographystyle{plain}
\bibliography{bibliography}

 \end{document}